\documentclass{article}
\usepackage[utf8]{inputenc}
\usepackage{mathrsfs}
\usepackage{mathtools}
\usepackage{authblk}
\usepackage[colorinlistoftodos]{todonotes}   
\usepackage{amsmath}
\usepackage{amssymb}
\usepackage{amsthm}
\numberwithin{equation}{section}

\usepackage[left=3cm, right=3cm]{geometry}
\usepackage{makeidx}
\makeindex

\usepackage{multicol}
\usepackage{tikz,pgf}
\usepackage{tikz-cd}
\usepackage{float}
\usepackage{xcolor}
\usepackage[all,knot,arc,color,web]{xy}
\usepackage{graphicx}
\xyoption{curve}
\usepackage{bbold}
\usepackage{fancyhdr}
\usepackage{subfigure}
\usepackage{subcaption}

\theoremstyle{plain}
\newtheorem{theorem}{Theorem}[section]
\newtheorem{proposition}[theorem]{Proposition}
\newtheorem{corollary}[theorem]{Corollary}
\newtheorem{lemma}[theorem]{Lemma}
\newtheorem{definition}[theorem]{Definition}

\theoremstyle{definition}
\newtheorem{remark}[theorem]{Remark}

\newtheorem{assumption}[theorem]{Assumption}

\begin{document}

\title{Area spectral rigidity for 
axially \\ symmetric symplectic billiard tables}

\author[1]{Luca Baracco\thanks{\texttt{baracco@math.unipd.it}}}
\author[1]{Olga Bernardi\thanks{\texttt{obern@math.unipd.it}}}
\author[2]{Alessandra Nardi\thanks{\texttt{alessandra.nardi@ist.ac.at}}}
\affil[1]{Dipartimento di Matematica Tullio Levi--Civita, Università di Padova, Italy.}
\affil[2]{IST Austria,
Klosterneuburg, Austria.}

\date{ }

\maketitle

\begin{abstract}
\noindent We prove that any finitely smooth axially symmetric strictly convex domain, with everywhere positive curvature and sufficiently close to an ellipse is area spectrally rigid. This means that any area-isospectral family of domains in this class is necessarily equi-affine. We use techniques, adapted to symplectic billiards, inspired to the paper by J. De Simoi, V. Kaloshin and Q. Wei \cite{DS}.
\end{abstract}

\section{Introduction}
Symplectic billiards were introduced by P. Albers and S. Tabachnikov in 2018 \cite{AlbTab}and were subsequently studied by various authors who investigated integrability \cite{BaBe}, \cite{BaBeNa2}, its Mather's $\beta$-function \cite{BaBeNa} and area spectral rigidity \cite{FSV}, \cite{DT}. \\ \\
\indent Let $\Omega$ be a strictly convex planar domain with fixed orientation. Through the whole paper we suppose that $\partial \Omega$ has everywhere positive curvature. Given $x,y,z \in \partial \Omega$, the symplectic billiard map $\Phi$ sends the chord $\Bar{xy}$ to the chord $\Bar{yz}$ if the segment $z-x$ is parallel to the tangent line in $y$. $\Phi$ turns out to be a twist map, preserving an area form, with generating function 
$$\omega(x, y) = \det(x,y).$$ Consequently, to every closed trajectory $\{x_j\}_{j = 0}^q$ of $\Phi$ in $\Omega$ ($x_0 = x_q$), it corresponds the action 
$\sum_{j = 0}^{q-1} \omega (x_j,x_{j+1})$. In particular, if the periodic trajectory winds once around $\partial \Omega$, then this action is the double of the area of the convex polygon inscribed in $\partial \Omega$ with vertices $\{x_j\}_{j = 0}^q$. We recall that $\Phi$ commutes with affine transformations of the plane. \\ \\
\indent In such a setting, we define the area spectrum for the symplectic billiard in $\Omega$ as the set of positive real numbers 
$$\mathcal{A}(\Omega)=\mathbb{N}\lbrace \text{action of all closed trajectories of $\Phi$}\rbrace\cup \mathbb{N}\lbrace A_\Omega\rbrace,$$ 
where $A_\Omega$ is the area of $\Omega$. We underline the difference of the area spectrum and the so-called marked area spectrum $\mathcal{MA}(\Omega)$, which is defined as the map that associates to any $1/q$ the maximal area of periodic trajectories having rotation number $1/q$. Marked area spectrum and Mather's $\beta$-function for the symplectic billiard dynamics are then related by
$$\beta(1/q):=-\frac{2}{q}\mathcal{MA}(\Omega)(1/q).$$
We remark that the sign minus in the above equality comes from the use of the generating function $-\omega(x,y)$ since the Mather’s $\beta$-function is defined by using minimal, instead of maximal,
trajectories. \\ \\
\indent S. Marvizi and R. Melrose’s theory of interpolating Hamiltonians \cite{MaMel} applied to the symplectic billiard map implies that $\Phi$ equals to the time-one flow of a Hamiltonian vector field composed
with a smooth map fixing pointwise the boundary of the phase-space at all orders. We refer also to \cite{Glu}[Section
2.1] for a detailed proof in the general case of (strongly) billiard-like maps. As an outcome, this result gives the following expansion at $0$ of the Mather's $\beta$ function:
$$\beta(1/q)\sim\frac{\beta_1}{q}+\frac{\beta_3}{q^3}+\frac{\beta_5}{q^5}+\frac{\beta_7}{q^7}+\dots$$
as $q\to\infty$. In such a setting, many questions corresponding to what it is known for Birkhoff billiards arise. \\ \\
\indent As suggested in \cite{AlbTab}[Remark 2.11], a first natural question is if there is a differential operator whose spectrum is precisely the sequence $\beta_1, \beta_3, \beta_5, \beta_7, \ldots$ For Birkhoff billiards, it is well-known that the corresponding length spectrum is closely related to the spectrum of the Laplace operator (see e.g. \cite{MaMel} and also \cite{PetSto}[Chapter 7]), but this problem is completely open for symplectic billiards. \\ \\
\noindent The coefficient $\beta_1$ is clearly $-2A_{\Omega}$. The two coefficients $\beta_3$ and $\beta_5$ were established in \cite{Lud} and the next one $\beta_7$ in \cite{BaBeNa}; in both cases, the techniques are from affine differential geometry and apply also to outer billiards. Consequently, another interesting question is whether it is possible to determine --maybe through a recursive formula-- the higher terms $\beta_9$, $\beta_{11}$... In the Birkhoff billiard case, this problem has been solved by A. Sorrentino \cite{Sor}. \\ \\
\noindent Another direction of study is the existence of area spectrally rigid classes of domains; this means that any family of domains in one of these classes with the same area spectrum is necessarily equi-affine. For Birkhoff billiards, the corresponding length spectral rigidity has been investigated by J. De Simoi, V. Kaloshin, and Q. Wei in \cite{DS}. In particular, they proved that the class of finitely smooth strictly convex axially symmetric domains, sufficiently close to a circle is length spectrally rigid, that is any
length-isospectral (or dynamically spectrally rigid) family of domains in this class is necessarily isometric. Their technique can be resumed in the next steps. First of all, they reduce any family of axially symmetric domains to a normalized one by using uniquely isometries. Secondly, for every $q \ge 2$, they prove the existence of marked locally maximizing axially symmetric $q$-periodic orbits of rotation number $1/q$. Then, if the given family is length-isospectral, to each orbit as above they associate a linear operator and rephrase the main result to the claim that this operator is injective. Finally, to prove the injectivity, they use precise estimates for a modification of the Lazutkin coordinates. \\ \\
\indent The intent of the present paper is to rephrase the result by De Simoi, Kaloshin and Wei for symplectic billiards. As explained above, in this setting length spectral rigidity is replaced by area spectral rigidity; moreover --since the symplectic billiard map commutes with affine transformations-- isometries are replaced by equi-affinities. Precise notions are contained in the next definition. \\ \\
\noindent \textbf{DEFINITION.} For $r \ge 2$, let $\mathcal{D}^r$ be the set of strictly convex domains with $C^{r+1}$ boundary and everywhere positive curvature. A  $C^1$ parametric family of domains $(\Omega_\tau)_{|\tau|\leq1} \subset \mathcal{D}^r$ is said to be 
\begin{itemize}
\item[$\bullet$] area-isospectral if $\mathcal{A}(\Omega_\tau)=\mathcal{A}(\Omega_{\tau'})$ for every $\tau,\tau'\in[-1,1];$
\item[$\bullet$] equi-affine if there exists a family $(\mathcal{B}_\tau)_{|\tau|\leq1}$ of affinities 
\begin{equation*}
        %\begin{split}
            \mathcal{B}_\tau: \mathbb{R}^2 \to\mathbb{R}^2, \quad 
            x \mapsto B_\tau x+b_\tau
        %\end{split}
    \end{equation*} with $B_\tau\in SL (2, \mathbb{R})$ and $b_\tau\in\mathbb{R}^2$ such that 
    $$\Omega_\tau=\mathcal{B}_\tau\Omega_0 \qquad \forall \tau \in [-1,1].$$ 
\end{itemize}
%The aim of this paper is to construct two different classes of domains for which the area spectral rigidity holds. The first class --analogously to the Birkhoff case-- is the set of finitely smooth strictly convex axially symmetric domains, with everywhere positive curvature and sufficiently close to an ellipse; the second class is the set of finitely smooth strictly convex centrally symmetric domains, with everywhere positive curvature, sufficiently close to an ellipse and ??? (according to the next definition). {\color{red}{We premise a definition to the statements of the two theorems.
\noindent The main result of the present paper is then given by the next theorem. 
\\ \\
\noindent   
\textbf{THEOREM.} Let $\mathcal{M}$ be the set of strictly convex domains with sufficiently (finitely) smooth boundary, everywhere positive curvature, axial symmetry, and sufficiently close to an ellipse. Let $(\Omega_{\tau})_{|\tau|\le 1}$ be a $C^1$ parameter family of domains in $\mathcal{M}$. If $(\Omega_{\tau})_{|\tau|\le 1}$ is area-isospectral
%that is $\mathcal{A}(\Omega_\tau}) = \mathcal{A}(\Omega_\tau'})$ for every $\tau, \tau' \in [-1,1]$. 
then $(\Omega_{\tau})_{|\tau|\le 1}$ is equi-affine. \\ \\
\noindent As explained for the seminal result in \cite{DS}, the proof of this theorem passes through the construction for every $\Omega_{\tau} \in \mathcal{M}$ of marked axially symmetric periodic orbits of rotation number $1/q$ ($q \ge 3$). Studying the variation of their areas in terms of $\tau$ results sufficient to obtain information on the family $(\Omega_{\tau})_{|\tau|\le 1}$ itself and then to prove the theorem when $\Omega_0$ is sufficiently close to an ellipse. Clearly --in general-- it is not always possible to construct a selected family of periodic orbits providing information on the variation of the spectrum. In light of the previous theorem, this further direction of investigation naturally arises. Does in general the marked area spectrum uniquely determine --up to affinities-- a domain? Or equivalently, does the Mather's $\beta$-function uniquely determine --up to affinities-- a domain? \\ \\
\noindent We underline that the recent paper \cite{FSV} by C. Fierobe, A. Sorrentino, and A. Vig presents closely related results, with an approach similar to the original one of De Simoi, Kaloshin and Wei. In both papers, spectral rigidity is addressed by showing that isospectrality implies the Fourier coefficients of the deformation function lie in the kernel of a certain operator. For domains close to a circle (in the Birkhoff case) and to an ellipse (in the symplectic case), this operator is injective, providing trivial deformations. Our approach is instead more direct, since we derive precise inequalities for the Fourier coefficients of the deformation function and conclude that, for domains sufficiently close to an ellipse, these inequalities imply zero Fourier coefficients. In particular, from the detailed proof of the theorem, it is clear both the necessity of the geometrical hypotheses on the family of domains and the role of the normalizations, which essentially compensate for the lack of periodic orbits of period smaller than $3$. \\ \\
\indent The paper is organized as follows. In Section \ref{S2} we define the area spectrum and area-isospectral and equi-affine families of domains. Moreover, we introduce the class of domains for which the area spectral rigidity holds. In Section \ref{S3} we first introduce the normalization for such a class of domains, which is essentially the same as the one proposed in \cite{DS}[Section 3]; clearly, by the affine equivariance of the symplectic billiard map, we reduce any family to a normalized one by using uniquely affinities. Finally, we define the infinitesimal deformation function for the normalized family of domains and we rephrase the area isospectrality in terms of a linear system involving the deformation function. In Section \ref{S5} we recall some preliminaries in affine differential geometry and we prove a technical lemma --refinement of Proposition 3.3 in \cite{BaBeNa}-- which will be fundamental in the proof of the theorem. In fact, our proofs are performed by using the affine parametrization. In particular, if $\{ s_j\}_{j=0}^q$ are the ordered affine parameters corresponding to a simple periodic trajectory for the symplectic billiard, we use --as in \cite{BaBeNa}-- the Taylor expansions of both $s_j$ and of $\lambda_j := s_j - s_{j-1}$. Section \ref{sezione finale} is devoted to the detailed proof of the main theorem. \\ \\
\indent The authors would like to express sincere thanks for the anonymous referee whose comments and suggestions helped improve and clarify the manuscript.

\section{Preliminaries and statement of the result}\label{S2}
Let $\Omega$ be a strictly convex planar domain with smooth boundary $\partial \Omega$ and fixed positive counterclockwise orientation. We assume that $\Omega$ contains the origin. By the strict convexity, for every point $x \in \partial \Omega$ there exists a unique (opposite) point $x^*$ such that $T_x \partial \Omega = T_{x^*} \partial \Omega$. We refer to
$$\mathcal{P} := \{(x, y) \in \partial \Omega \times \partial \Omega : x < y < x^*\}$$
as the (open, positive) phase-space and we define the symplectic billiard map as follows (see \cite{AlbTab}[Page 5]):
$$\Phi : \mathcal{P} \to \mathcal{P}, \quad (x, y) \mapsto (y, z)$$
where $z$ is the unique point satisfying 
$$z-x \in T_y \partial \Omega.$$
We notice that $\Phi$ is continuous and can be continuously extended to $\bar{\mathcal{P}}$, so that $\Phi(x,x) = (x,x)$ and $\Phi(x,x^*) = (x^*,x)$ (these are the unique $2$-periodic orbits of $\Phi$). Moreover, see \cite{AlbTab}[Section 2] for exhaustive details, the symplectic billiard map turns out to be a twist map, preserving an area form, with generating function 
$$\omega(x, y) = \det(x,y)$$ 
so that
$$\Phi(x,y) = (y,z) \Leftrightarrow \frac{d}{dy} \left[ \omega(x,y) + \omega(y,z) \right] = 0.$$
We refer also to \cite{BaBe}, \cite{BaBeNa}, and \cite{BaBeNa2} for recent advances in symplectic billiards. \\
\noindent To every closed trajectory $\{x_j\}_{j = 0}^q$ of $\Phi$ in $\Omega$ ($x_0 = x_q$), it corresponds the action $\sum_{j = 0}^{q-1} \omega (x_j,x_{j+1})$. In particular, if the periodic trajectory winds once around $\partial \Omega$, then $\frac{1}{2} \sum_{j = 0}^{q-1} \omega (x_j,x_{j+1})$ is the area of the convex polygon inscribed in $\partial \Omega$ with vertices $\{x_j\}_{j = 0}^q$.  
\begin{definition} 
The area spectrum for the symplectic billiard in $\Omega$ is the set of positive real numbers 
$$\mathcal{A}(\Omega)=\mathbb{N}\lbrace \text{action of all closed trajectories of $\Phi$}\rbrace\cup \mathbb{N}\lbrace A_\Omega\rbrace,$$ 
where $A_\Omega$ is the area of $\Omega$.
\end{definition}
\noindent We underline the difference between the area spectrum and the so-called marked area spectrum, which is defined as the map that associates to any $1/q$, $q \ge 3$, the maximal area of (simple) convex polygons with $q$ vertices inscribed in $\partial \Omega$. See e.g. \cite{AlbTab}[Beginning of Section 2.5.2] and \cite{BaBeNa}[Definition 1.4]. We stress that the substantial difference of the two spectra is that the marked area spectrum preserves the information on the rotation number. \\
~\newline
For $r \ge 2$, let $\mathcal{D}^r$ be the set of strictly convex domains with $C^{r+1}$ boundary, {\textit{with everywhere positive curvature}}. The proof of this property of the area spectrum is the same as the one in \cite{DS}[Lemma 4.1].
\begin{lemma} \label{lemma Sard} Let $r \ge 2$. For any domain $\Omega\in\mathcal{D}^r$, $\mathcal{A}(\Omega)$ has zero Lebesgue measure.
\end{lemma} 
%\begin{proof} Fixed $q\geq3$, let us define 
%$$\omega_q(\mathbf{x}) := \sum_{j = 0}^{q-1} \omega(x_j,x_{j+1})$$
%where $\mathbf{x} := \{ x_j \}_{j = 0}^q$ with $x_0 \le x_1 \le \ldots \le x_{q-1} \le x_q = x_0$. The critical points of $\omega_q$ correspond to $q$-periodic orbits for the symplectic billiard map, and their actions is the set of critical values of $\omega_q$. Since the symplectic billiard map is $C^r$ and --by assumption-- $r\geq2$, by Sard's Theorem we have that the set of actions of $q$-periodic orbits has zero Lebesgue measure. The statement is then obtained by taking the countable union over $q\geq3$. 
%\end{proof}
\noindent Given $\Omega \in \mathcal{D}^r$, we assume that the affine perimeter of $\partial \Omega$ is normalized to $1$ and we indicate by
$$\mathbb{T} \ni s \mapsto \gamma_{\Omega}(s) \in \partial \Omega$$ 
the $C^{r}$ affine arc-length parametrization of $\partial \Omega$, where $\mathbb{T} = \mathbb{R}/\mathbb{Z} = [0,1] / \sim$ identifying $0 \sim 1$. \\
\noindent In particular, for $\Omega_0,\Omega_1 \in \mathcal{D}^r$, let 
\begin{equation*}
    d(\Omega_0,\Omega_1) := \sum_{k=0}^{r}\sup_{s\in[0,1]}\|\gamma_{\Omega_1}^{(k)}(s)-\gamma_{\Omega_0}^{(k)}(s)\|.
\end{equation*}
The next definition corresponds to Definition 2.9 in \cite{DS}. 
\begin{definition} \label{dist cerchio} Given $\Omega\in\mathcal{D}^r$ with unitary affine perimeter, let $D$ be the disk of affine perimeter $1$ tangent to $\Omega$ at the point $\gamma_{\Omega}(0)$. For $\delta>0$, $\Omega$ is said to be $\delta$-close to the circle if 
\begin{equation} \label{delta}
d(\Omega,D)\leq\delta.
\end{equation}
A domain $\Omega\in\mathcal{D}^r$ of arbitrary affine perimeter is said to be $\delta$-close to the circle if its rescaling of unitary affine perimeter satisfies (\ref{delta}). 
\end{definition}
\begin{assumption} \label{no hp} From now on, every domain $\Omega \in \mathcal{D}^r$ is assumed to be: 
\begin{itemize}
\item[$(a)$] parametrized by the affine parameter; 
\item[$(b)$] normalized to unitary affine perimeter. 
\end{itemize}
\end{assumption}
\noindent Given a $C^1$ parametric family of domains $(\Omega_\tau)_{|\tau|\leq1}\subset \mathcal{D}^r$, we use the notation
$$\gamma:[-1,1]\times\mathbb{T} \to\mathbb{R}^2, \quad (\tau,s) \mapsto \gamma(\tau, s) := \gamma_{\Omega_{\tau}}(s).$$
By hypothesis, $\gamma(\cdot,s)$ is $C^1$ for every $s \in \mathbb{T} $ and $\gamma(\tau,\cdot)$ is $C^{r}$ for every $\tau\in[-1,1]$. \\
\noindent Notice that $\gamma$ and $\Tilde{\gamma}$ parametrize the same family of domains $(\Omega_\tau)_{|\tau|\leq1} \subset \mathcal{D}^r$ if and only if there exists a $C^1$ family of $C^{r}$ circle diffeomorphisms $\tilde{s}:[-1,1]\times\mathbb{T} \to\mathbb{T}$ such that 
\begin{equation} \label{equivalent families}
\gamma(\tau,s)=\tilde{\gamma}(\tau,\tilde{s}(\tau,s) = s + s_0(\tau)) 
\end{equation} 
(or equivalently, $\tilde{\gamma}(\tau,\tilde{s}) = \gamma(\tau,s(\tau,\tilde{s}))$, where $s$ denotes the inverse of $\tilde{s}$). Two parametrizations corresponding to the same family of domains are said to be equivalent. \\
\noindent Let keep in mind the definition of area-isospectral and equi-affine families of domains given in the Introduction and fix a subset $\mathcal{M} \subset \mathcal{D}^r$. 
\begin{definition}
A domain $\Omega \in \mathcal{M}$ is called area spectrally rigid in $\mathcal{M}$ if any area-isospectral family of domains $(\Omega_\tau)_{|\tau|\leq1}\subset\mathcal{M}$ with $\Omega_0=\Omega$ is necessarily equi-affine.
\end{definition}
\begin{remark} The aim of the present paper is to construct a class of domains for which the area spectral rigidity holds. As a consequence, we underline that Assumption \ref{no hp} is not restrictive because if $(\Omega_{\tau})_{|\tau| \le 1} \subset \mathcal{D}^r$ is area-isospectral then the affine perimeter of $\partial \Omega_{\tau}$ is constant. In fact, under the hypothesis of area-isospectrality, the marked area spectrum for the symplectic billiard in $\Omega_{\tau}$ is independent on $\tau$. Equivalently, the corresponding formal Taylor expansion at $0$ of the Mather's $\beta$-function is independent on $\tau$. To conclude, it is then sufficient to remind that the term of order $3$ is exactly $\lambda^3/6$ (see e.g. \cite{Lud}[Theorem 1]), where $\lambda$ is the affine perimeter of the symplectic billiard table.
\end{remark}

\noindent In order to state the main result, we start by recalling that a domain $\Omega\in\mathcal{D}^r$ is said to be \textit{axially symmetric} if there exists a line $\Delta\subset\mathbb{R}^2$ such that $\Omega$ is invariant under the reflection along $\Delta$. By convexity, $\Delta \cap \partial \Omega$ is given by two points. Chosen (arbitrarily) one of these points, we refer to it as the marked point of $\partial \Omega$ and to the other as the auxiliary point of $\partial \Omega$. \\
The next lemma --which is a consequence of the axial symmetry of $\Omega$-- corresponds to Lemma 4.3 in \cite{DS}. 
\begin{lemma} \label{max axially}
Let $\Omega\in\mathcal{D}^r$ be an axially symmetric domain. For every $q\geq3$ there exists an axially symmetric periodic orbit of rotation number $1/q$ passing through the marked point of $\partial\Omega$ and of maximal area among all axially symmetric $q$-periodic orbits passing through the marked point. %This orbit will be denoted by $S_q(\Omega)$. 
\end{lemma}

\begin{proof} We recall that $\Omega$ is parametrized in affine arc-length by $\gamma(s)$; moreover, since $\lambda=1$, the marked and the auxiliary points can be conventionally fixed to be $\gamma(0)$ and $\gamma(1/2)$ respectively. We distinguish even and odd period and we use the same arguments as in the proof of Lemma 4.3 in \cite{DS}. \\
\indent For $k\geq2$, let $q=2k$ even. In such a case, the desired $q$-periodic orbit passes through the marked and the auxiliary points. In fact, once fixed $s_0=0$ and $s_k=1/2$, we look for $(s_1,\dots,s_{k-1})$ in the compact set $0=s_0\leq s_1\leq\dots\leq s_{k-1}\leq s_k=1/2$, maximizing 
\begin{equation}\label{even}
\sum_{j=0}^{k-1}\omega(\gamma(s_j),\gamma(s_{j+1})).
\end{equation}
The maximum $(0,\Bar{s}_1,\dots,\Bar{s}_{k-1},1/2)$ must satisfy $0<\Bar{s}_1<\dots<\Bar{s}_{k-1}<1/2$. Moreover,
$$0 = \partial_1 \omega(\gamma(\Bar{s}_j),\gamma(\Bar{s}_{j+1})) + \partial_2 \omega(\gamma(\Bar{s}_{j-1}),\gamma(\Bar{s}_j)) \qquad \forall j=1,\dots,k-1.$$ 
Defining $\Bar{s}_{2k-j}=-\Bar{s}_j$ for every $j=1,\dots k-1$, we obtain the desired $2k$-periodic orbit. See $(a)$ in Figure \ref{figura 1}. \\
% \begin{center}
% \includegraphics[scale=0.3]{even.jpeg}    
% \end{center}
\indent For $k\geq1$, let $q=2k+1$ odd. In such a case, differently from above, we fix only $s_0=0$ and we look for $(s_1,\dots,s_k)$ in the compact set $0=s_0\leq s_1\leq\dots\leq s_k \le 1/2$, maximizing 
\begin{equation}\label{odd} 
\sum_{j=0}^{k-1}\omega(\gamma(s_j),\gamma(s_{j+1}))+\frac{1}{2}\omega(\gamma(s_k),\gamma(-s_{k})).
\end{equation} Again, the maximum $(0,\bar{s}_1,\dots,\bar{s}_k)$ must satisfy $0 <\Bar{s}_1<\dots<\Bar{s}_{k-1}<\Bar{s}_k<1/2$. Moreover
\begin{equation*}
        \begin{split}
            0&=\partial_1 \omega(\gamma(\Bar{s}_j),\gamma(\Bar{s}_{j+1}))+\partial_2 \omega(\gamma(\Bar{s}_{j-1}),\gamma(\Bar{s}_j)) \qquad \forall j=1,\dots,k-1, \\
            %0&=\partial_2 \omega(\Bar{s}_{k-1},\Bar{s}_{k})+\frac{1}{2}\partial_1 \omega(\Bar{s}_{k},-\Bar{s}_{k})-\frac{1}{2}\partial_1 A(\Bar{s}_{k},-\Bar{s}_{k})\\
            0&=\partial_2 \omega(\gamma(\Bar{s}_{k-1}),\gamma(\Bar{s}_{k}))+\partial_1 \omega(\gamma(\Bar{s}_{k}),\gamma(-\Bar{s}_{k})).
        \end{split}
\end{equation*}
The $(2k+1)$-periodic orbit of the statement is then obtained by defining $\Bar{s}_{2k+1-j}=-\Bar{s}_j$ for every $j=1,\dots k-1$. See $(b)$ in Figure \ref{figura 1}.
\end{proof}
\begin{figure}[H]
\centering
    \begin{subfigure}[]
        \centering
        \includegraphics[scale=0.25]{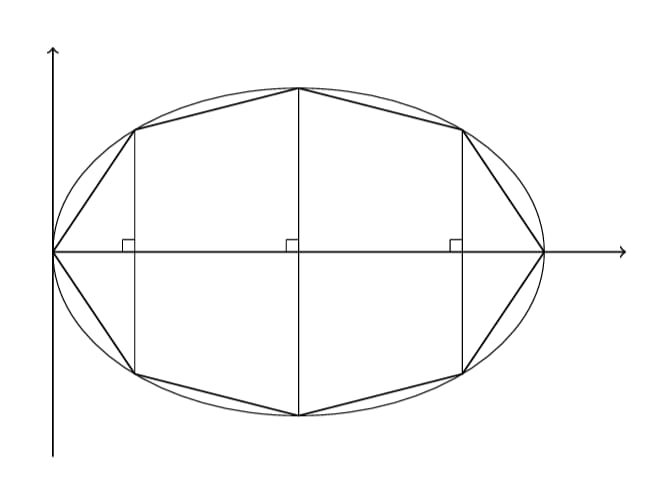}
        %\subcaption{$q=2k$}
    %\label{fig:}
    \end{subfigure}
    \begin{subfigure}[]
        \centering
        \includegraphics[scale=0.25]{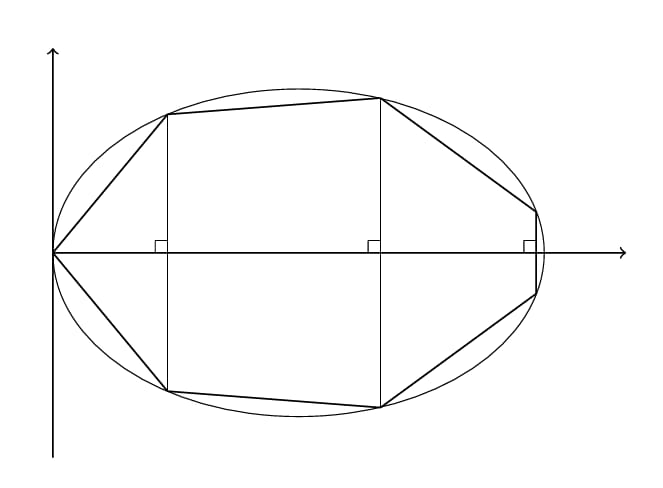}
        %\subcaption{$q=2k+1$}
    %\label{fig:}
    \end{subfigure}
    \caption{Axially symmetric periodic orbits of even period $(a)$ and odd period $(b)$.}
    \label{figura 1}
\end{figure}
\noindent 
\noindent In the sequel, we denote by $\mathcal{AS}^r_{\delta}$ the set of axially symmetric domains in $\mathcal{D}^r$ which are $\delta$-close to the circle, in accord with Definition \ref{dist cerchio}.

\section{Normalization and infinitesimal deformation function}\label{S3}

We now introduce the normalization. Given a family $(\Omega_{\tau})_{|\tau| \le 1} \subset \mathcal{D}^r$ of axially symmetric domains, parametrized by $\gamma$, the associated normalized family $(\tilde{\Omega}_{\tau})_{|\tau| \le 1}$ is then constructed as follows:
\begin{enumerate}
    %\item apply a rescaling to $\Omega$ of unitary affine perimeter $\lambda=1$;
    \item[$(a)$] By a translation, impose the marked point to be $\gamma(\tau,0) = (0,0)$ for every $\tau \in [-1,1]$;
    \item[$(b)$]  by a rotation, impose the axis of symmetry of every $\Omega_{\tau}$ to be on $x>0$; 
    \item[$(c)$]  finally, by a unitary diagonal transformation, impose the auxiliary point of every $\partial \Omega_{\tau}$ to be $\left(2 (2\pi)^{-3/2},0\right)$.
\end{enumerate} 
%\begin{figure}[H]
%\centering
 %       \includegraphics[scale=0.8]{norm_as.pdf}
  %  \caption{Normalization of axially symmetric domains.}
   % \label{norm_as}
%\end{figure}\alessandra{Fatemi sapere se va bene la figura}
%\indent Before proceeding with the introduction of the normalization in the area-isospectral, Radon case, we need to premise the next proposition.
%\begin{proposition}
%Let $\Omega \in \mathcal{D}^r$ be a Radon domain, parametrized by $\gamma_{\Omega}$. There exists $\bar{s} \in \mathbb{T}$ such that the $4$-periodic orbit passing through $\gamma_{\Omega}(\bar{s})$ corresponds to parameters 
%$\lbrace \bar{s},\bar{s}+1/4,\bar{s}+1/2,\bar{s}+3/4\rbrace.$
%\end{proposition}
%\begin{proof}
%Since $\Omega \in \mathcal{D}^r$ is Radon, it is well defined the continuous map $\mathbb{T} \ni s \mapsto \phi(s) \in \mathbb{T}$ associating to each $s$ the parameter $\phi(s)$ of the next (with respect to the fixed counterclockwise orientation) vertex of the $4$-periodic orbit passing through $\gamma_{\Omega}(s)$, that is $\{ \gamma_{\Omega}(s), \gamma_{\Omega}(\phi(s)), \gamma_{\Omega}(s + 1/2), \gamma_{\Omega}(\phi(s) + 1/2)\}$. Suppose $\phi (0) < 1/4$. Since $\phi^2(0) = 1/2$, there necessarily exists $\bar{s} \in (0,\phi(0))$ such that $\phi(\bar{s}) = \bar{s} + 1/4$, which is the desired result. The case $\phi (0) > 1/4$ can be treated equivalently. 
%\end{proof}
\noindent We underline that all transformations used in points $(a)$--$(c)$ of the normalization do not change the (unitary) affine perimeter, and that circles and ellipses coincide with this normalization. Moreover, given a family $(\Omega_{\tau})_{|\tau| \le 1} \subset \mathcal{D}^r$ of axially symmetric domains, we have used affinities to construct $(\tilde{\Omega}_{\tau})_{|\tau| \le 1}$. As a consequence, $\mathcal{A}(\Omega_\tau) =\mathcal{A}(\Tilde{\Omega}_\tau)$ for every $\tau \in [-1,1]$ and therefore $(\Omega_\tau)_{|\tau|\leq1}$ is area isospectral if and only if $(\Tilde{\Omega}_\tau)_{|\tau|\leq1}$ is area isospectral. Consequently, the Theorem stated in the Introduction can be rephrased as follows. From now on --when it is clear from the context-- we omit the tilde to indicate a normalized family.
\begin{theorem}\label{THMC} Let $r=6$. There exists $\delta>0$ such that any domain $\Omega\in\mathcal{AS}^r_\delta$ is area spectrally rigid in $\mathcal{AS}^r_\delta$. In other words, any normalized area-isospectral family of domains $(\Omega_\tau)_{|\tau|\leq1}\subset \mathcal{AS}^r_\delta$ with $\Omega_0 = \Omega$ is necessarily constant.
\end{theorem}
\noindent Given a normalized family $(\Omega_\tau)_{|\tau|\leq1}\subset\mathcal{D}^r$ of domains in $\mathcal{AS}^r_\delta$, parametrized by $\gamma$, we define the infinitesimal deformation function $n_\gamma(\tau,s)$ as  
\begin{equation}
    n_\gamma(\tau,s):=\omega(\partial_\tau\gamma(\tau, s), T_{\gamma}(\tau,s)),
\end{equation}
where $T_{\gamma}(\tau,s)$ is the unit tangent vector to $\partial \Omega_{\tau}$ at the point $\gamma(\tau,s)$, $T_{\gamma}(\tau,s) := \frac{\partial_s\gamma(\tau,s)}{\| \partial_s\gamma(\tau,s) \|}$. We refer to \cite{DS}[Page 245] for the formula of the infinitesimal deformation function for (normalized) axially symmetric Birkhoff billiard tables. By the regularity properties of $\gamma$, $n_\gamma$ is $C^0$ in $\tau$ and $C^{r-1}$ in $s$. 
\begin{remark} \label{lunedi}
We finally notice that, as a consequence of the normalization, for a normalized family $(\Omega_{\tau})_{|\tau| \le 1} \subset \mathcal{AS}^r_\delta$, $n_{\gamma}(\tau,0) = n_{\gamma}(\tau,1/2) = 0$ for every $|\tau| \le 1$.
We stress that the above properties will be fundamental in order to conclude the proof of the main theorem; as it will become clear later, this normalization will compensate for the lack of periodic orbits of period smaller than $3$. 
\end{remark}
\noindent The statement and the proof of the next lemma is the same of Lemma 3.3 in \cite{DS}. %We sketch the proof for reader's convenience. 
\begin{lemma}
Let $(\Omega_\tau)_{|\tau|\leq1}\subset\mathcal{D}^r$ be a normalized family of domains in $\mathcal{AS}^r_\delta$, parametrized by $\gamma$. Then
    \begin{enumerate}
        \item[$(a)$] for any other parametrization $\Tilde{\gamma}$ as in (\ref{equivalent families}), if $n_\gamma(\tau,\cdot) \equiv 0$ for some $\tau \in [-1,1]$ then $n_{\Tilde{\gamma}}(\tau,\cdot) \equiv 0 $ for the same $\tau \in [-1,1]$;
        \item[$(b)$] $n_\gamma(\tau,s)=0$ for all $(\tau,s)\in [-1,1]\times\mathbb{T}$ if and only if $(\Omega_\tau)_{|\tau|\leq1}$ is a constant family.
    \end{enumerate}
\end{lemma}
%\begin{proof} $(a)$ Let consider two parametrizations $\gamma$ and $\Tilde{\gamma}$ of $(\Omega_\tau)_{|\tau|\leq1}$, so that $\gamma(\tau,s)=\tilde{\gamma}(\tau,\tilde{s}(\tau,s))$. By differentiating in $\tau$, we obtain $$\partial_\tau\gamma(\tau,s)=\partial_\tau\Tilde{\gamma}(\tau,\Tilde{s}(\tau,s))+\partial_s\Tilde{\gamma}(\tau,\Tilde{s}(\tau,s))\partial_\tau\Tilde{s}(\tau,s),$$
%where $\partial_\tau\Tilde{s}(\tau,s) = \partial_{\tau}s_0(\tau)$. By taking the determinant between $T_{\gamma}(\tau,s) = T_{\tilde{\gamma}}(\tau,\tilde{s}(\tau,s))$, we immediately obtain that $n_{\gamma}(\tau,s)=n_{\tilde{\gamma}}(\tau, \tilde{s}(\tau,s))$. \\
%$(b)$ Observe that if $n_\gamma(\tau,s)=0$, then necessarily $\partial_\tau\gamma(\tau,s)$ is proportional to the vector $\partial_s\gamma(\tau,s)$. By the regularity properties of $\gamma$, $d\gamma(\tau,s)$ has rank $1$ everywhere. Therefore, by the Constant Rank Theorem, the image of $\gamma$ is a $1$-dimensional manifold that can be parametrized by $\gamma(0,\cdot)$, which implies $\Omega_\tau=\Omega_0$ for every $|\tau|\leq1$. Other implication follows from $(a)$.
%\end{proof}
\noindent By using the lemma above, we can rephrase Theorem \ref{THMC} in terms of the infinitesimal deformation function. 
\begin{theorem}\label{THM3} Let $r=6$. There exists $\delta>0$ such that if $(\Omega_\tau)_{|\tau|\leq1}$ is a normalized area-isospectral family of domains in $\mathcal{AS}^r_\delta$, then $n_\gamma\equiv0$ for every parametrization $\gamma$.
\end{theorem}
%\section{Area isospectrality and deformation function}\label{S4}
\noindent Consider a family of domains $(\Omega_\tau)_{|\tau|\leq1}\subset\mathcal{D}^r$. For every $q\geq3$, if $\mathbf{s}(\tau) = \lbrace s_i(\tau)\}_{i=0}^q$ are the affine parameters for a $q$-periodic sequence corresponding to points in $\partial \Omega_\tau$, we indicate 
$$\omega_q(\tau;\mathbf{s}(\tau)) := \sum_{j = 0}^{q-1} \omega (\gamma(\tau,s_j(\tau)),\gamma(\tau,s_{j+1}(\tau))).$$ 
Clearly, if $\mathbf{s}(\tau)$ gives a $q$-periodic orbit, then $\omega_q(\tau;\mathbf{s}(\tau))$ is its action. \\
In the sequel, for the sake of simplicity, we omit the dependence on $\tau$ of $\mathbf{s}$ and $s_i$. Let $(\Omega_\tau)_{|\tau|\leq1}\subset \mathcal{AS}^r_\delta$ be a normalized family of domains. For every $q \ge 3$, let $\mathbf{s} = \{ s_j \}_{j = 0}^q$ be the affine parameters corresponding to the ordered vertices of an axially symmetric $q$-periodic orbit for the symplectic billiard in $\Omega_\tau$, passing through the marked point of $\partial \Omega_{\tau}$. Moreover, let
\begin{equation} \label{max}
    [-1,1] \ni \tau \mapsto \Delta_q(\tau) := \max_{\mathbf{s}} \, \omega_q(\tau;\mathbf{s}) \in \mathbb{R}
\end{equation}
be the function associating to each domain $\Omega_\tau$ twice the area of the $q$-periodic orbit constructed in Lemma \ref{max axially}. More precisely, $\Delta_q(\tau)$ is twice the maximal area among all axially symmetric $q$-periodic orbits for the symplectic billiard in $\Omega_{\tau}$, passing through the marked point of $\partial \Omega_\tau$.
\begin{proposition} \label{subdiff} 
Let $(\Omega_\tau)_{|\tau|\leq1}\subset\mathcal{AS}^r_\delta$ be a normalized area-isospectral family of domains parametrized by $\gamma$. If $\bar{\mathbf{s}} = (\bar{s}_0, \ldots , \bar{s}_q)$ realizes the maximum (\ref{max}), that is
$\omega_q(\tau;\mathbf{\bar{s}}) = \Delta_q(\tau)$, then 
\begin{equation} \label{d tau = 0}
\partial_\tau \omega_q(\tau,\Bar{\mathbf{s}}) =\sum_{j=0}^{q-1}
\|\gamma(\tau,\Bar{s}_{j+1})-\gamma(\tau,\Bar{s}_{j-1})\|
n_\gamma(\tau,\Bar{s}_j) = 0.
\end{equation}
\end{proposition}
\begin{proof}
We first prove that $\partial_\tau \omega_q(\tau,\Bar{\mathbf{s}}) = 0$. By definition, $\Delta_q(\tau)\in\mathcal{A}(\Omega_\tau)$. Moreover, since the family $(\Omega_\tau)_{|\tau|\leq1}\subset\mathcal{AS}^r_\delta$ is assumed to be $C^1$ parametric, $\Delta_q$ is a continuous function. By these facts and Lemma \ref{lemma Sard}, it follows that if $(\Omega_\tau)_{|\tau|\leq1} \subset \mathcal{AS}^r_\delta$ is area-isospectral then $\Delta_q$ is necessarily constant. We recall that $\Delta_q(\tau)$ is --by definition-- twice the maximal area among all axially symmetric $q$-periodic orbits for the symplectic billiard in $\Omega_{\tau}$, passing through the marked point of $\partial \Omega_\tau$. Consequently, if $\Bar{\mathbf{s}}$ realizes the maximum at $\tau$ then
$$\omega_q(\tau',\Bar{\mathbf{s}}) \leq\Delta_q(\tau') = \Delta_q(\tau) = \omega_q(\tau;\bar{\mathbf{s}}) \qquad \forall \tau'\in[-1,1],$$
where --in the first equality-- we have used the fact that (in the area-isospectral case) $\Delta_q$ is constant. Since the family $(\Omega_\tau)_{|\tau|\leq1}\subset\mathcal{AS}^r_\delta$ is assumed to be $C^1$ parametric, $\tau \mapsto \omega_q(\tau;\bar{\mathbf{s}})$ is $C^1$ and the statement immediately follows. Finally, we verify the explicit expression for $\partial_\tau \omega_q(\tau,\Bar{\mathbf{s}})$ by a direct computation:
\begin{equation*}
\begin{split}
\partial_\tau \omega_q(\tau,\Bar{\mathbf{s}})&=\sum_{j=0}^{q-1}\partial_\tau\omega(\gamma(\tau,\Bar{s}_j),\gamma(\tau,\Bar{s}_{j+1}))\\ 
&=\sum_{j=0}^{q-1}\left[\omega(\partial_\tau\gamma(\tau,\Bar{s}_j),\gamma(\tau,\Bar{s}_{j+1}))+\omega(\gamma(\tau,\Bar{s}_j),\partial_\tau\gamma(\tau,\Bar{s}_{j+1}))\right]\\ 
&=\sum_{j=0}^{q-1}\omega(\partial_\tau\gamma(\tau,\Bar{s}_j),\gamma(\tau,\Bar{s}_{j+1})-\gamma(\tau, \Bar{s}_{j-1}))\\ 
&=\sum_{j=0}^{q-1}\|\gamma(\tau,\Bar{s}_{j+1})-\gamma(\tau,\Bar{s}_{j-1})\|\omega(\partial_\tau\gamma(\tau,\Bar{s}_j),T_{\gamma}(\tau,\Bar{s}_j)) \\ &=\sum_{j=0}^{q-1}
\|\gamma(\tau,\Bar{s}_{j+1})-\gamma(\tau,\Bar{s}_{j-1})\| 
n_\gamma(\tau,\Bar{s}_j). \end{split}
\end{equation*}
where, in the fourth equality, we have used the symplectic billiard dynamics: 
$$\gamma(\tau,\Bar{s}_{j+1})-\gamma(\tau, \Bar{s}_{j-1}) = \|\gamma(\tau,\Bar{s}_{j+1})-\gamma(\tau,\Bar{s}_{j-1})\| \, T_{\gamma}(\tau,\Bar{s}_j) \qquad \forall j = 1, \ldots, q-1.$$
\end{proof}

\section{Preliminary results in affine differential geometry}\label{S5}

This section is devoted to collecting some notions and results in affine differential geometry, which will be useful in the next section. We recall that the affine arc length and the affine perimeter are given respectively by
$$s(t) = \int^t_0 \kappa^{1/3}(r)dr \qquad 0 \le t \le l$$
and
$$\lambda = \int^l_0 \kappa^{1/3}(r)dr$$
where $t$ is the arc length and $\kappa(t)$ the (ordinary) curvature of $\partial \Omega$. As in the previous sections, let $\gamma(s)$ be the affine arc length parametrization of $\partial \Omega$. Then --see \cite{Bu}[Section 3]-- $\gamma$ is characterized by:
$$\omega (\gamma^{\prime}, \gamma^{\prime \prime}) = 1, \qquad \omega (\gamma^{\prime}, \gamma^{\prime \prime \prime}) = 0$$
and
$$k := \omega(\gamma'',\gamma''')$$
is the affine curvature of $\partial \Omega$ (remind that the affine curvature of the circle of affine unitary perimeter is $(2\pi)^2$). Moreover, differentiating $\omega (\gamma^{\prime}, \gamma^{\prime \prime \prime}) = 0$ and using the definition of $k$, we obtain 
$$\gamma''' = -k \gamma'.$$
%We underline that the affine curvature of the circle $C$ of affine length $\lambda=1$ is $(2\pi)^2$, and it can be easily computed by mean of the affine isoperimetric equality for the circle, $\lambda=8\pi^2Area(C)$.\\
\noindent We underline that $\|\gamma'(s)\|=\rho(t(s))$, where $\rho$ and $t$ are the (usual) ray of curvature and arc length, respectively. From now on, we suppose $\lambda = 1$ (see Assumption \ref{no hp}). \\
\noindent The next technical lemma is a refinement of Proposition 3.3 in \cite{BaBeNa}. These coordinates and their expansions are the analogue, for symplectic billiards, of those computed by V.F. Lazutkin for the usual Birkhoff billiards \cite{Laz}.
\begin{lemma}
For $q \ge 3$, let $\{ \gamma(s_j)\}_{j = 0}^q$ be the ordered vertices of a periodic trajectory for the symplectic billiard with rotation number $1/q$ and $s_0 = 0$ and $s_q = 1$. If
$$\lambda_j := s_j - s_{j-1} \qquad \forall j = 1, \ldots, q$$ 
then
 \begin{equation} \label{PRIMA}
 \lambda_j=\frac{1}{q}-\frac{1}{30q^3}\int_0^1 k(s) ds+\frac{1}{30q^3}k\left(\frac{j}{q}\right)-\frac{1}{60q^4}k'\left(\frac{j}{q}\right)+O\left(\frac{1}{q^5}\right),
 \end{equation}
 where, for some constant $C > 0$ independent on $q \ge 3$:
$$\left\| O\left(\frac{1}{q^5}\right) \right\| \le C \, \frac{\| k' \|_{C^2}}{q^5}.$$
Moreover
  \begin{equation}\label{sj}
         s_j=\frac{j}{q} +\frac{1}{30q^2}\int^\frac{j}{q}_0 k(s) ds -\frac{j}{30q^3}\int_0^1 k(s)ds %-\frac{1}{30q^3}\left(k\left(\frac{j}{q}\right)-k(0)\right)
         +O\left(\frac{1}{q^4}\right).
     \end{equation}
\end{lemma}
\begin{proof} We first prove formula (\ref{PRIMA}). We start recalling that, by Propositions 3.2 and 3.3 in \cite{BaBeNa}, it holds that
\begin{equation}\label{lambaj+1}
        \begin{split}
            \lambda_{j+1}&=\lambda_j+\frac{1}{30}k'(s_j)\lambda_j^4+O(\lambda_j^6) \qquad \text{and} \\
            \lambda_j&=\frac{1}{q}-\frac{1}{30q^3}\int_0^1 k(s) ds + \frac{1}{30q^3}k\left(\frac{j}{q}\right)+\frac{\sigma(j,q)}{q^4}
        \end{split}
    \end{equation}
    with $\sigma(j,q)=O(1)$. 
    \begin{remark}
    We stress that, since we are applying the Implicit function Theorem, in order to obtain the first formula in (\ref{lambaj+1}), we need to ask the affine parametrization to be at least $C^6$, that is $r = 6$.
    \end{remark}
\noindent    Combining the two equalities, we obtain 
    \begin{equation*}\label{lambdaj}
        \begin{split}
            \frac{\sigma(j+1,q)}{q^4}&=\frac{\sigma(j,q)}{q^4}+\frac{1}{30q^3}\left(k\left(\frac{j}{q}\right)-k\left(\frac{j+1}{q}\right)+\frac{1}{q}k'\left(\frac{j}{q}\right)\right)+O\left(\frac{1}{q^6}\right)\\
            &=\frac{\sigma(j,q)}{q^4}-\frac{1}{60q^5}k''\left(\frac{j}{q}\right)+O\left(\frac{1}{q^6}\right),
              \end{split}
    \end{equation*}
    so that
\begin{equation} \label{sigmaj}
\sigma(j,q)=\sigma(1,q)- \frac{1}{60q} \sum_{i=1}^{j-1}k''\left(\frac{i}{q}\right)+O\left(\frac{1}{q}\right) = \sigma(1,q)-\frac{1}{60}\left(k'\left(\frac{j}{q}\right)-k'\left(\frac{1}{q}\right)\right)+O\left(\frac{1}{q}\right),
\end{equation}
where --in the last equality-- we have used the fact that by Taylor expansion:
 %   By Taylor expansion and direct computation, 
    $$\frac{1}{q} \sum_{i=1}^{j-1}k''\left(\frac{i}{q}\right) -\int_{\frac{1}{q}}^\frac{j}{q}k''(s)ds = \frac{1}{q} \sum_{i=0}^{j-1}k''\left(\frac{i}{q}\right) - k'\left( \frac{j}{q} \right) + k'\left( \frac{1}{q} \right) = O\left(\frac{1}{q}\right)
    .$$
    %which implies \begin{equation*}
     %   \sigma(j,q)=\sigma(0,q)-\frac{1}{60}\left(k'\left(\frac{j}{q}\right)-k'(0)\right)+O\left(\frac{1}{q}\right).
    %\end{equation*}
Therefore:
\begin{equation*} 
 \lambda_j=\frac{1}{q}-\frac{1}{30q^3}\int_0^1 k(s) ds+\frac{1}{30q^3}k\left(\frac{j}{q}\right) + \frac{\sigma(1,q)+ \frac{1}{60} k'\left( \frac{1}{q}\right)}{q^4} -\frac{1}{60q^4}k'\left(\frac{j}{q}\right)+O\left(\frac{1}{q^5}\right).
 \end{equation*}
 We continue by summing up the last equality for $j = 1, \ldots, q$. Recalling that $\sum_{j = 1}^{q} \lambda_j = 1$, we obtain:
    \begin{equation*}
        \begin{split}
            0=&\frac{1}{30q^3}\sum_{j=1}^{q}k\left(\frac{j}{q}\right)-\frac{1}{30q^2}\int_0^1 k(s)ds-\frac{1}{60q^4}\sum_{j=1}^{q}k' \left(\frac{j}{q}\right) +\frac{\sigma(1,q)+\frac{1}{60}k'\left(\frac{1}{q}\right)}{q^3} + O \left(\frac{1}{q^4}\right)\\
            =& \frac{1}{30q^2}\sum_{j=1}^{q}\int_\frac{j}{q}^\frac{j+1}{q}k\left(\frac{j}{q}\right)-k(s)ds-\frac{1}{60q^4} \sum_{j=1}^{q} k'\left(\frac{j}{q}\right)+\frac{\sigma(1,q)+\frac{1}{60}k'\left(\frac{1}{q}\right)}{q^3}+O\left(\frac{1}{q^4}\right).
        \end{split}
    \end{equation*}
    Since $$\int_\frac{j}{q}^\frac{j+1}{q}k\left(\frac{j}{q}\right)-k(s)ds=-\frac{1}{2q^2}k'\left(\frac{j}{q}\right)+O\left(\frac{1}{q^3}\right) \qquad \text{and} \qquad  \frac{1}{q} \sum_{j=1}^{q}k'\left(\frac{j}{q}\right) = O\left(\frac{1}{q}\right),$$
    we finally obtain 
    \begin{equation*}
        \frac{\sigma(1,q)+\frac{1}{60}k'\left(\frac{1}{q}\right)}{q^3} = O\left(\frac{1}{q^4}\right)
    \end{equation*}
By using (\ref{sigmaj}), we immediately obtain formula \eqref{PRIMA}. \\
\indent Formula (\ref{sj}) comes by summing up expansion \eqref{PRIMA} for $ i = 1, \ldots , j$:
    \begin{equation*}
    \begin{split}
        s_j &= \sum_{i=1}^{j}\lambda_i = \frac{j}{q} - \frac{j}{30q^3}\int_0^1 k(s)ds + \frac{1}{30 q^3} \sum_{i = 1}^{j} k \left( \frac{i}{q} \right) - \frac{1}{60q^4}\sum_{i=1}^{j}k'\left(\frac{i}{q}\right)+O\left(\frac{1}{q^4}\right). \\
        %& =\frac{j}{q}-\frac{j}{30q^3}\int_0^1 k(s)ds +\frac{1}{30q^2}\int_0^\frac{j}{q}k(s)ds-\frac{1}{30q^4}\sum_{i=1}^{j}k'\left(\frac{i}{q}\right)+O\left(\frac{1}{q^4}\right).
        \end{split}
    \end{equation*}
Since now, by Taylor's expansion
$$\frac{1}{30q^3}\sum_{i=1}^{j}k\left(\frac{i}{q}\right)=\frac{1}{30q^2}\int_0^\frac{j}{q}k(s)ds+\frac{1}{60q^4}\sum_{i=1}^{j}k'\left(\frac{i}{q}\right)+O\left(\frac{1}{q^4}\right),$$
%Similarly as above --by Taylor expansion-- we get
    %$$-\sum_{i=0}^{j-1}\frac{1}{q}k'\left(\frac{i}{q}\right)+\int_0^\frac{j}{q}k'(s)ds=\sum_{i=0}^{j-1}\int_\frac{i}{q}^\frac{i+1}{q}k'(s)-k'\left(\frac{i}{q}\right)ds=\sum_{i=0}^{j-1}\frac{1}{q^2}k''\left(\frac{i}{q}\right)+O\left(\frac{1}{q^3}\right),$$
  %  $$\frac{1}{30q^3}\sum_{i=1}^{j}\frac{1}{q}k'\left(\frac{i}{q}\right)=\frac{1}{30q^3}\left(k\left(\frac{j+1}{q}\right)-k\left(\frac{1}{q}\right)\right)+O\left(\frac{1}{q^4}\right)=\frac{1}{30q^3}\left(k\left(\frac{j}{q}\right)-k(0)\right)+O\left(\frac{1}{q^4}\right)$$
formula \eqref{sj} immediately follows. 
\end{proof}

\section{Proof of Theorem \ref{THM3}} \label{sezione finale}
%\alessandra{Devo sistemare questo inizio e il remark}
For $\delta>0$, if $(\Omega_\tau)_{|\tau|\leq1} \subset \mathcal{AS}^r_{\delta}$ is a normalized area-isospectral family of domains, parametrized by $\gamma$, $n_{\gamma}(\tau,\cdot)$ is an even function. In particular, for every $\tau\in[-1,1]$ we can write the Fourier expansion of $n_{\gamma}(\tau,s)$ in the basis $\{ e^{2\pi i k s} \}_{k\in\mathbb{Z}}$:
    $$n_\gamma(\tau, s)=\sum_{|k|\geq1}\hat{n}_k(\tau)e^{2\pi iks}$$ 
    where $\hat{n}_k(\tau)$ denotes the $k$-th Fourier coefficient of $n_\gamma(\tau,\cdot)$ and by symmetry $\hat{n}_k=\hat{n}_{-k}$. Since $n_\gamma$ has zero-average, $\hat{n}_0=0$.
    We remind that, as a consequence of the normalization (see Remark \ref{lunedi}) $n_{\gamma}(\tau,0) = n_{\gamma}(\tau,1/2) = 0$ for every $|\tau| \le 1$. 
\begin{remark}
 Theorem \ref{THM3} is especially easy to verify when $\Omega_0$ is a circle. In this special case, for every $q \ge 3$ the axially symmetric periodic orbit of Lemma \ref{max axially} is given by $\bar{\mathbf{s}} = \left\{ \frac{j}{q} \right\}_{j = 0}^q$. Consequently, from Proposition \ref{subdiff} we have that
$$\partial_\tau \omega_q(\tau,\Bar{\mathbf{s}})\mid_{\tau=0} =\sum_{j=0}^{q-1}
\bigg\|\gamma\left(0,\frac{j+1}{q}\right)-\gamma\left(0,\frac{j-1}{q}\right)\bigg\|
\, n_\gamma\left(0,\frac{j}{q}\right) = 0.
$$
Since $\bigg\|\gamma\left(0,\frac{j+1}{q}\right)-\gamma\left(0,\frac{j-1}{q}\right)\bigg\|$ is constant for each $j=0,\ldots,q-1$, we get 
\begin{equation}\label{m1}
\sum_{j=0}^{q-1}n_\gamma\left(0,\frac{j}{q}\right)=\sum_{|k|\geq1}\sum_{j=0}^{q-1}\hat{n}_k(0)e^{\frac{2\pi ikj}{q}}=2q\sum_{k\geq1}\hat{n}_{qk}(0) = 0.
\end{equation}
Moreover, since $n_{\gamma}$ is $C^{r-1}$ in $s$, we have 
\begin{equation}\label{m3}
|\hat{n}_{q}(0)| \le \frac{C}{q^{r-1}} \quad \forall q\ge 1   
\end{equation}
for a suitable constant $C > 0$, uniform in $q\ge 1$, which depends only on $n_\gamma$. By \eqref{m1} and \eqref{m3} we have that
\begin{equation} \label{m2}
    |\hat{n}_q(0)|=\left| \sum_{k=2}^{+\infty}\hat{n}_{kq}(0) \right|\le \sum_{k=2}^{+\infty} \frac{C}{(kq)^{r-1}}=\frac{C}{q^{r-1}}(\zeta(r-1)-1)\quad \forall q \ge 3.
\end{equation}
We can therefore repeat the same argument with $C$ replaced by $C(\zeta(r-1)-1)$ and then proceed by induction, obtaining
\begin{equation*}
    |\hat{n}_q(0)|\le \frac{C}{q^{r-1}}\bigl((\zeta(r-1)-1)\bigr)^k,\quad \forall k \ge 1 \text{ and } \forall q\ge 3.
\end{equation*}
Since $0 < \zeta(r-1)-1 < 1$, this implies that $\hat{n}_q(0)=0$ for every $q \ge 3$. In addition, by the normalization condition, $n_\gamma(0,0) = 0 = n_\gamma(0,1/2)$, which yields $\hat{n}_1(0) = \hat{n}_2(0) = 0$. This gives $n_\gamma(0,\cdot)\equiv0$ and therefore the desired statement, $n_\gamma \equiv 0$.

%This verifies Theorem \ref{THM3} in the case in which the initial domain is the circle.
%\begin{comment} \\

%\noindent In order to verify Theorem \ref{THM4} in the case of the circle, notice that for $q\geq4$ even, both $\Bar{\mathbf{s}}$ and $\Bar{\mathbf{s}_c}=\left\{ \frac{j}{q}+\frac{1}{4q} \right\}_{j = 0}^q$ are maximizing $q$-periodic orbit satisfying assumptions of Remark \ref{erorb}.\\
%\noindent Again since  $\bigg\|\gamma\left(0,\frac{j+1}{q}\right)-\gamma\left(0,\frac{j-1}{q}\right)\bigg\|=\bigg\|\gamma\left(0,\frac{j+1}{q}+\frac{1}{4q}\right)-\gamma\left(0,\frac{j-1}{q}+\frac{1}{4q}\right)\bigg\|$ are constant for each $j=0,\ldots,q-1$, we get  
%\begin{equation*}
%\sum_{j=0}^{q-1}n_\gamma\left(0,\frac{j}{q}\right)=\sum_{|k|\geq1}\sum_{j=0}^{q-1}\hat{n}_{2k}(0)e^{\frac{4\pi ikj}{q}}=q\sum_{|k|\geq1}\hat{n}_{2kq}(0) =q\sum_{k\geq1}2Re(\hat{n}_{2kq}(0))  0,
%\end{equation*}
%\begin{equation*}
%\sum_{j=0}^{q-1}n_\gamma\left(0,\frac{j}{q}+\frac{1}{4q}\right)=\sum_{|k|\geq1}\sum_{j=0}^{q-1}\hat{n}_{2k}(0)e^{\frac{4\pi ikj}{q}}e^{\frac{\pi ik}{q}}=q\sum_{|k|\geq1}i\hat{n}_{2kq}(0) =-q\sum_{k\geq1}2Im(\hat{n}_{2kq}(0))  0,
%\end{equation*}
%implying that $\hat{n}_{k}(0)=0$ for every $k\geq4$ even. Moreover, from normalization, $n_\gamma(0,0) = 0 = n_\gamma(0,1/4)$ so that $\hat{n}_2(0) = \hat{n}_{-2}(0) = 0$. This gives $n_\gamma(0,\cdot)\equiv0$ and therefore the desired statement for $\mathcal{ER}^r_\delta$, $n_\gamma \equiv 0$.\\
%$$n_\gamma(0,s)=\hat{n}_1(0)e^{2\pi is}+\hat{n}_2(0)e^{4\pi is}$$ and by  conditions we know that , which gives  and in particular
%\end{comment}

\end{remark}
\noindent Before entering into the details of the proof of Theorem \ref{THM3}, we premise a technical lemma.
\begin{lemma} Let $(\Omega_\tau)_{|\tau|\leq1}\subset\mathcal{D}^r$ be a normalized family of domains in $\mathcal{AS}^r_\delta$, parametrized by $\gamma$. For $q \ge 3$, let $\mathbf{s} = \{ s_j \}_{j = 0}^q$ be the affine parameters corresponding to ordered vertices of a periodic trajectory for the symplectic billiard for the domain $\Omega_0$, with rotation number $1/q$. Then
\begin{equation} \label{generale}
\partial_\tau \omega_q(\tau,\mathbf{s})\mid_{\tau=0}
%=\sum_{j=0}^{q-1}\|\gamma(s_{j+1})-\gamma(s_{j-1})\|n_\gamma\left(0,s_j\right)
=\sum_{j=0}^{q-1}n_\gamma\left(0,s_j\right)\rho(t(s_j))^\frac{1}{3}\left[\frac{2}{q} + \frac{1}{15q^3} \left( k\left(\frac{j}{q}\right) - \int_0^1 k(s)ds \right) -\frac{1}{3q^3}k(s_j)+O\left(\frac{1}{q^5}\right)\right],
\end{equation}
where $\rho$ denotes the (usual) ray of curvature of the domain $\Omega_0$.
\end{lemma}
\begin{proof} From Proposition \ref{subdiff}, we know that
     $$\partial_\tau \omega_q(\tau,\Bar{\mathbf{s}})\mid_{\tau=0}=\sum_{j=0}^{q-1}
\|\gamma(0,s_{j+1})-\gamma(0,s_{j-1})\| 
n_\gamma(0,s_j),$$
so that we need to detect $\|\gamma(0,s_{j+1})-\gamma(0,s_{j-1})\|$. Naming $\gamma(0,s)=\gamma(s)$ and recalling that $\lambda_j=s_j-s_{j-1}$:
$$
\begin{aligned}
& \gamma\left(s_{j+1}\right)-\gamma\left(s_j\right)=\gamma^{\prime}\left(s_j\right) \lambda_{j+1}+\frac{\gamma^{\prime \prime}\left(s_j\right)}{2} \lambda_{j+1}^2+\frac{\gamma^{\prime \prime \prime}\left(s_j\right)}{6} \lambda_{j+1}^3+\frac{\gamma^{\prime \prime \prime \prime}\left(s_j\right)}{24} \lambda_{j+1}^4+O\left(\lambda_{j+1}^5\right), \\
& \gamma\left(s_{j-1}\right)-\gamma\left(s_j\right)=-\gamma^{\prime}\left(s_j\right) \lambda_j+\frac{\gamma^{\prime \prime}\left(s_j\right)}{2} \lambda_j^2-\frac{\gamma^{\prime \prime \prime}\left(s_j\right)}{6} \lambda_j^3+\frac{\gamma^{\prime \prime\prime \prime}\left(s_j\right)}{24} \lambda_j^4+O\left(\lambda_j^5\right).
\end{aligned}
$$
Consequently, by expansion \eqref{PRIMA} and the fact that $\lambda_{j+1}-\lambda_j=O(1/q^4)$ (see first formula in \eqref{lambaj+1}), the norm of their difference is
$$\|\gamma(s_{j+1})-\gamma(s_{j-1})\|=\bigg\|\gamma'(s_j)(\lambda_{j+1}+\lambda_j)+\frac{\gamma'''(s_j)}{6}(\lambda_{j+1}^3+\lambda_j^3)+O(1/q^5)\bigg\|.$$
Plugging formula \eqref{PRIMA} into the one above, we obtain that
$$\|\gamma(s_{j+1})-\gamma(s_{j-1})\| =$$
$$\|\gamma'(s_j)\| \left[\frac{2}{q}-\frac{1}{30q^3}\left(2\int_0^1 k(s)ds-k\left(\frac{j}{q}\right)-k\left(\frac{j+1}{q}\right)\right)-\frac{1}{3q^3}\frac{\|\gamma'''(s_j)\|}{\|\gamma'(s_j)\|}-\frac{1}{60q^4}\left(k'\left(\frac{j}{q}\right)+k'\left(\frac{j+1}{q}\right)\right)+O\left(\frac{1}{q^5}\right)\right] =$$
$$\rho(t(s_j))^\frac{1}{3} \left[\frac{2}{q}-\frac{1}{30q^3}\left(2\int_0^1 k(s)ds-k\left(\frac{j}{q}\right)-k\left(\frac{j+1}{q}\right)\right)-\frac{1}{3q^3}k(s_j)-\frac{1}{60q^4}\left(k'\left(\frac{j}{q}\right)+k'\left(\frac{j+1}{q}\right)\right)+O\left(\frac{1}{q^5}\right)\right].$$
\noindent By Taylor expansions, the order $4$ term cancels out and we finally get
$$\|\gamma(s_{j+1})-\gamma(s_{j-1})\|=\rho(t(s_j))^\frac{1}{3}\left[\frac{2}{q} + \frac{1}{15q^3} \left( k\left(\frac{j}{q}\right) - \int_0^1 k(s)ds \right) -\frac{1}{3q^3}k(s_j)+O\left(\frac{1}{q^5}\right)\right],$$
which is the desired result.
\end{proof}
\begin{corollary}\label{beta} Let $(\Omega_\tau)_{|\tau|\leq1}\subset\mathcal{D}^r$ be a normalized family of domains in $\mathcal{AS}^r_\delta$, parametrized by $\gamma$ and $\delta$-close to the circle. For $q \ge 3$, let $\mathbf{s} = \{ s_j \}_{j = 0}^q$ be the affine parameters corresponding to ordered vertices of a periodic trajectory for the symplectic billiard for the domain $\Omega_0$, with rotation number $1/q$. Then
$$\partial_\tau \omega_q(\tau,\mathbf{s})\mid_{\tau=0}
=\sum_{j=0}^{q-1}n_\gamma\left(0,s_j\right)\rho(t(s_j))^\frac{1}{3}\left[\frac{2}{q} -\frac{(2\pi)^2}{3q^3} + \frac{\beta(j/q)}{q^3} +O\left(\frac{\delta}{q^5}\right)\right],$$
where $\rho$ denotes the (usual) ray of curvature of the domain $\Omega_0$ and
$$\beta(j/q) = \frac{1}{15} \left[k\left(\frac{j}{q}\right)- \int_0^1 k(s)ds\right] + \frac{1}{3} \left( \left(2\pi\right)^2-k\left(\frac{j}{q}\right)\right).$$
\end{corollary}
\begin{proof}
Notice that, under the hypothesis of $\delta$-closeness to the circle, the reminder of expansion (\ref{generale}) is $O\left( \frac{\delta}{q^5} \right)$. Then it is sufficient to substitute expansion (coming from (\ref{sj})):
$$k(s_j) = k\left(\frac{j}{q}\right) + O \left( \frac{\delta}{q^2}\right)$$
in (\ref{generale}) and sum and subtract $\frac{(2\pi)^2}{3q^3}$ (remind that $(2\pi)^2$ is the affine curvature of the circle of unitary affine perimeter).  %The properties of symmetry and smallness of the function $\beta$ follows immediately from the fact that $\Omega_0\in\mathcal{AS}^r_{\delta}$.
\end{proof}

\subsection{Proof of Theorem \ref{THM3}}
Consider $(\Omega_\tau)_{|\tau|\leq1}\subset\mathcal{AS}^r_\delta$ a family of normalized area-isospectral axially symmetric domains. For $q \ge 3$, let $\Bar{\mathbf{s}} = \{ s_j \}_{j = 0}^q$ be the affine parameters corresponding to the ordered vertices of the maximizing axially symmetric $q$-periodic orbit for the symplectic billiard in $\Omega_0$, as constructed in Lemma \ref{max axially}. By Proposition \ref{subdiff}, $\partial_\tau \omega_q(\tau,\Bar{\mathbf{s}})\mid_{\tau=0}=0$. This means that, as a consequence of Corollary \ref{beta}, 
$$\sum_{j=0}^{q-1}n_\gamma\left(0,s_j\right)\rho(t(s_j))^\frac{1}{3}\left[\frac{2}{q}-\frac{(2\pi)^2}{3q^3}+\frac{\beta(j/q)}{q^3}+O\left(\frac{\delta}{q^5}\right)\right] = 0.$$ 
Since, by hypothesis, $\rho(t(s))>0$ for every $s\in\mathbb{T}$, we indicate 
$$ u(s) = n_\gamma\left(0,s\right)\rho(t(s))^\frac{1}{3}$$ 
and investigate on
\begin{equation*}
\sum_{j=0}^{q-1}u(s_j)\left[\frac{2}{q}-\frac{(2\pi)^2}{3q^3}+\frac{\beta(j/q)}{q^3}+O\left(\frac{\delta}{q^5}\right)\right] = 0.
\end{equation*}
Since $u(s)$ is even and --by isospectral hypothesis-- with zero-average, we substitute its Fourier expansion in the basis $(\cos{(2\pi ks)})_{k\ge1}$. We get
     \begin{equation}\label{serie}
          \sum_{j=0}^{q-1}\sum_{k\geq1}\hat{u}_k\cos\left(2\pi ks_j\right)\left[\frac{2}{q}-\frac{(2\pi)^2}{3q^3}+\frac{\beta(j/q)}{q^3}+O\left(\frac{\delta}{q^5}\right)\right] = 0.
     \end{equation}         
Formula \eqref{sj} in the closeness to the circle assumption reads 
\begin{equation} \label{ultima fatica}
s_j=\frac{j}{q}+\frac{\alpha(j/q)}{q^2}+O\left(\frac{\delta}{q^4}\right), \qquad  
    \alpha(j/q):=\frac{1}{30}\int_0^\frac{j}{q}k(s)ds-\frac{j}{30q}\int_0^1k(s)ds. %\quad \text{and}\quad \Tilde{\alpha}(j/q)=\frac{1}{30}\left(k\left(\frac{j}{q}\right)-k(0)\right),
\end{equation}
We remark that, in the axially symmetric case, the function $\alpha$ is odd and of class $C^r$ and the function $\beta$ of Corollary \ref{beta} is even and of class $C^{r-1}$. \\
Substituting in (\ref{serie}) the above expansion of $s_j$, we obtain (up to rename $\beta$):
     \begin{equation*}
 \frac{2}{q} \sum_{j=0}^{q-1}\sum_{k\geq1}\hat{u}_k\left[\cos\left(2\pi k\left(\frac{j}{q}+\frac{\alpha(j/q)}{q^2}+O\left(\frac{\delta}{q^4}\right)\right)\right)\left(1-\frac{2\pi^2}{3q^2}+\frac{\beta(j/q)}{q^2}+O\left(\frac{\delta}{q^4}\right)\right)\right] = 0.
\end{equation*} 
Recall that, from the $\delta$--closeness of the domain to the circle (see the second formula in (\ref{ultima fatica})) $\|\alpha\|_{C^2}\leq\delta$. We now use Taylor expansion for the term involving the cosine:
\begin{equation*}
\cos\left(2\pi k\left(\frac{j}{q}+\frac{\alpha(j/q)}{q^2}+O\left(\frac{\delta}{q^4}\right)\right)\right)=\cos\left(\frac{2\pi kj}{q}\right)-2\pi k\sin{\left(\frac{2\pi jk}{q}\right)}\frac{\alpha(j/q)}{q^2}+O\left(\frac{\delta k}{q^4}\right)+O\left(\frac{\delta^2 k^2}{q^4}\right),  
 \end{equation*} 
% we get
% \begin{equation*}
% \begin{split}
%     0=\sum_{j=0}^{q-1}\sum_{k\geq1}\hat{u}_k&\left[\cos\left(\frac{2\pi kj}{q}\right)-2\pi k\sin{\left(\frac{2\pi jk}{q}\right)}\frac{\alpha(j/q)}{q^2}+O\left(\frac{\delta k}{q^4}\right)+O\left(\frac{\delta^2 k^2}{q^4}\right)\right]\cdot \\ &\cdot\left(1-\frac{2\pi^2}{3q^2}+\frac{\beta(j/q)}{q^2}+O\left(\frac{\delta}{q^4}\right)\right).
 %\end{split}
 %\end{equation*} 
and the Fourier expansions of $\alpha$ and $\beta$. We then obtain:
 \begin{equation*}
 \begin{split}
     &\sum_{k\geq1}\hat{u}_k \left[q\cdot\delta_{q|k}\left(1-\frac{2\pi^2}{3q^2}\right)+O\left(\frac{\delta}{q^3}\right)+ O\left(\frac{\delta^2k^2}{q^3}\right)+ O\left(\frac{\delta k}{q^3}\right)\right]+ \\ 
     & + \frac{1}{q^2}\hat{u}_k\sum_{j=0}^{q-1}\sum_{p\in\mathbb{Z}}-2\pi i k\left(\exp{\left(\frac{2\pi ijk}{q}\right)}-\exp{\left(-\frac{2\pi ijk}{q}\right)}\right)\alpha_p\exp{\left(\frac{2\pi ipj}{q}\right)} + \\ 
     & + \left(\exp{\left(\frac{2\pi ijk}{q}\right)}+\exp{\left(-\frac{2\pi ijk}{q}\right)}\right)\beta_p\exp{\left(\frac{2\pi ipj}{q}\right)} = \\ 
& = \sum_{k\geq1}\hat{u}_k\left[q\cdot\delta_{q|k}\left(1-\frac{2\pi^2}{3q^2}\right)+ O\left(\frac{\delta^2k^2}{q^3}\right)+ O\left(\frac{\delta k}{q^3}\right)\right] +\hat{u}_k\frac{1}{q}\sum_{s\in\mathbb{Z}}\left[\beta_{sq-k}+\beta_{sq+k}+2\pi ik(\alpha_{sq+k}-\alpha_{sq-k})\right] = 0.
\end{split}
\end{equation*}
Since $\alpha_p=-\alpha_{-p}$ and  $\beta_p=\beta_{-p}$, we finally get:
\begin{equation*} \label{ordine 4 e 2}
\sum_{k\geq1}\hat{u}_k\left[\left(1-\frac{2\pi^2}{3q^2}+\frac{\beta_0}{q^2}\right)\delta_{q|k}+\frac{2}{q^2}\sum_{\substack{s\in\mathbb{Z}\\ sq\neq k}}\left(\beta_{sq-k}-2\pi ik\alpha_{sq-k}%-2\pi ik\frac{\Tilde{\alpha}_{sq-k}}{q}
\right)+ O\left(\frac{\delta^2k^2}{q^4}\right)+ O\left(\frac{\delta k}{q^4}\right)\right] = 0,
\end{equation*}
or, equivalently (for all $q\geq3$):
\begin{equation} \label{ordine}
\sum_{k\geq1}\hat{u}_k\left[\left(1-\frac{2\pi^2}{3q^2}+\frac{\beta_0}{q^2}\right)\delta_{q|k}+\frac{2}{q^2}\sum_{\substack{0 \ne s\in\mathbb{Z}\\ sq\neq k}}\left(\beta_{sq-k}-2\pi ik\alpha_{sq-k}%-2\pi ik\frac{\Tilde{\alpha}_{sq-k}}{q}
\right)+ O\left(\frac{\delta^2k^2}{q^4}\right)+ O\left(\frac{\delta k}{q^4}\right)\right] = -\frac{2}{q^2} \sum_{k \ge 1} \hat{u}_k (\beta_k + 2\pi i k \alpha_k).
\end{equation}
\indent Since we prove that the right-hand side of equality (\ref{ordine}) is $0$ (we refer to the discussion at the end of the proof of the next theorem), Theorem \ref{THM3} becomes a straightforward consequence of the next lemma, whose proof takes up most of the section.
\begin{lemma} \label{liposuzione} If $u(s) = n_\gamma\left(0,s\right)\rho(t(s))^\frac{1}{3}$ solves 
\begin{equation} \label{questa}
\sum_{k\geq1}\hat{u}_k\left[\left(1-\frac{2\pi^2}{3q^2}+\frac{\beta_0}{q^2}\right)\delta_{q|k}+\frac{2}{q^2}\sum_{\substack{0 \ne s\in\mathbb{Z}\\ sq\neq k}}\left(\beta_{sq-k}-2\pi ik\alpha_{sq-k}%-2\pi ik\frac{\Tilde{\alpha}_{sq-k}}{q}
\right)+ O\left(\frac{\delta^2k^2}{q^4}\right)+ O\left(\frac{\delta k}{q^4}\right)\right] = 0
\end{equation}
for all $q \ge 3$, then $n_{\gamma} \equiv 0$.
\end{lemma}
\begin{proof}
For all $q \ge 3$, (\ref{questa}) equals to
$$-\hat{u}_q\left(1-\frac{2\pi^2}{3q^2}+\frac{\beta_0}{q^2}\right)=\sum_{k\geq2}\hat{u}_{kq}\left(1-\frac{2\pi^2}{3q^2}+\frac{\beta_0}{q^2}\right)+  \sum_{k\geq1}\hat{u}_k\left[ \frac{2}{q^2} \sum_{\substack{0 \ne s\in\mathbb{Z}\\ sq\neq k}}\left(\beta_{sq-k}-2\pi ik\alpha_{sq-k}\right)+ O\left(\frac{\delta^2k^2}{q^4}\right)+ O\left(\frac{\delta k}{q^4}\right)\right].$$
Notice that
$$0.26 < \left(1-\frac{2\pi^2}{3q^2}+\frac{\beta_0}{q^2} \right) < 1 \qquad \forall q \ge 3$$
and take the absolute value in the last equality. We then obtain:
 \begin{equation*}\begin{split}
     & |\hat{u}_q|\left(1-\frac{2\pi^2}{3q^2}+\frac{\beta_0}{q^2}\right)\leq \\
     & \le \sum_{k\geq2}|\hat{u}_{kq}|\left(1-\frac{2\pi^2}{3q^2}+\frac{\beta_0}{q^2}\right) + \frac{2}{q^2} \sum_{k\geq1}|\hat{u}_k| \sum_{\substack{0 \ne s\in\mathbb{Z} \\ sq\neq k}}\left|\beta_{sq-k}-2\pi ik\alpha_{sq-k}\right| + \sum_{k\geq1}|\hat{u}_k|\left|O\left(\frac{\delta^2k^2}{q^4}\right)+ O\left(\frac{\delta k}{q^4}\right)\right|.
     \end{split}
 \end{equation*} 
In the sequel, for $f \in C^{r-1}(\mathbb{T})$ with zero average, fixed an integer $1\leq t\leq r-1$, we indicate with $\|f\|_t$ the norm
\begin{equation} \label{norm4}
\|f\|_t:=\max_{k \in \mathbb{Z}} |\hat{f}_k\cdot k^{t}| \Rightarrow |\hat{f}_k| \le \frac{\|f\|_t}{|k|^{t}}\quad \forall k\neq0,
\end{equation}
where $\hat{f_k}$ denotes the $k$-th Fourier coefficient of $f$.\\
\noindent Closeness to the circle assumption implies
$$\|\alpha\|_t\leq A\delta, \qquad\|\beta\|_t\leq B\delta$$
for some uniform constants $A, B > 0$. Recalling that $u\in C^{r-1}$, and both functions $\alpha$ and $\beta$ are at least $C^{r-1}$, for $1\leq t\leq r-1$ we get  
\begin{equation}\label{norm}\begin{split}
    &|\hat{u}_q|\left(1-\frac{2\pi^2}{3q^2}+\frac{\beta_0}{q^2}\right)\leq \\
    & \le \sum_{k\geq2}\frac{\|u\|_t}{(kq)^{t}}\left(1-\frac{2\pi^2}{3q^2}+\frac{\beta_0}{q^2}\right)+ \frac{2}{q^2} \sum_{k\geq1}\frac{\|u\|_t}{k^{t}}\sum_{\substack{0 \ne s\in\mathbb{Z} \\ sq\neq k}}\left|\beta_{sq-k}-2\pi ik\alpha_{sq-k}\right| + \sum_{k\geq1}\frac{\|u\|_t}{k^{t}}\left|O\left(\frac{\delta^2k^2}{q^4}\right)+ O\left(\frac{\delta k}{q^4}\right)\right|\\
    &\leq\sum_{k\geq2}\frac{\|u\|_t}{(kq)^{t}}\left(1-\frac{2\pi^2}{3q^2}+\frac{\beta_0}{q^2}\right)+ \frac{2}{q^2} \sum_{k\geq1}\frac{\|u\|_t}{k^{t}}\sum_{\substack{0 \ne s\in\mathbb{Z}\\ sq\neq k}}\left(\frac{B\delta}{|sq-k|^{t}}+\frac{A\delta k}{|sq-k|^{t}}\right) + \sum_{k\geq1}\frac{\|u\|_t}{k^{t}}\left|O\left(\frac{\delta^2k^2}{q^4}\right)+ O\left(\frac{\delta k}{q^4}\right)\right|\\
    &\leq\sum_{k\geq2}\frac{\|u\|_t}{(kq)^{t}}\left(1-\frac{2\pi^2}{3q^2}+\frac{\beta_0}{q^2}\right)+ \frac{\bar{C} \|u\|_t }{q^2} \sum_{k\geq1}\sum_{\substack{0 \ne s\in\mathbb{Z}\\ sq\neq k}}\frac{\delta k}{|sq-k|^{t} {k^{t}}} + \sum_{k\geq1}\frac{\|u\|_t}{k^{t}}\left|O\left(\frac{\delta^2k^2}{q^4}\right)+ O\left(\frac{\delta k}{q^4}\right)\right|.
\end{split}
\end{equation}
where $\bar{C}=4\max\lbrace A,B\rbrace> 0$ .  Let us now look at each summand (the aim is showing that $\hat{u}_q = 0$ for every $q\geq3$).
\begin{itemize}
    \item[$(a)$] The first term satisfies
    \begin{equation*}
        \sum_{k\geq2}\frac{\|u\|_t}{(kq)^{t}}\left(1-\frac{2\pi^2}{3q^2}+\frac{\beta_0}{q^2}\right)\leq\frac{\|u\|_t}{q^{t}}(\zeta(t)-1)\left(1-\frac{2\pi^2}{3q^2}+\frac{\beta_0}{q^2}\right).
    \end{equation*}
    Notice that to get the convergence of the series we need to ask $t\geq2$. %We stress that $${\color{red}\zeta(t)-1<0.65\qquad \forall t\geq2.}$$
    \item[$(b)$] The third term satisfies
    \begin{equation*}
        \sum_{k\geq1}\frac{\|u\|_t}{k^{t}}\left|O\left(\frac{\delta^2k^2}{q^4}\right)+ O\left(\frac{\delta k}{q^4}\right)\right|\leq\frac{\|u\|_t}{q^{4}}(\delta^2K_1\zeta(t-2)+\delta K_2\zeta(t-1))
    \end{equation*}
    for some constants $K_1,K_2 > 0$. For this term, in order to assure the convergence of the series, we need to ask $t\geq4$.
    \item[$(c)$] The second summand requires more work, and we use arguments analogous to the ones in \cite{DS}[Proof of Lemma 5.3]. Up to $\frac{\bar{C} \delta \| u \|_t}{q^2}$, the term to study is
    \begin{equation*} \begin{split}
     &\sum_{\substack{0 \ne s\in\mathbb{Z}\\ sq\neq k}} \sum_{k\geq1} \frac{1}{k^{t-1} |sq-k|^{t}} \le 2 \sum_{\substack{s \ge 1 \\ sq\neq k}} \sum_{k\geq1} \frac{1}{k^{t-1} |sq-k|^{t}} = \\
     & = 2 \sum_{s \ge 1}  \left( \sum_{1 \le k < \frac{sq}{2}} \frac{1}{k^{t-1} (sq-k)^{t}} + \sum_{\frac{sq}{2} \le k < sq} \frac{1}{k^{t-1} (sq-k)^{t}} + \sum_{k >sq} \frac{1}{k^{t-1} |sq-k|^{t}}\right) \le \\
     & \le 2 \sum_{s \ge 1} \left(  \sum_{1 \le k < \frac{sq}{2}} \frac{2^t}{k^{t-1} s^t q^t} + \sum_{\frac{sq}{2} \le k < sq} \frac{2^{t-1}}{s^{t-1}q^{t-1}(sq-k)^t} + \sum_{k > sq} \frac{1}{s^{t-1}q^{t-1}|sq-k|^{t}} 
     \right) \le \\
     & \le 2 \left( \frac{2^t \zeta(t-1)\zeta(t)}{q^t} + \frac{2^{t-1} \zeta(t-1)\zeta(t)}{q^{t-1}} + \frac{\zeta(t-1)\zeta(t)}{q^{t-1}} \right) = \frac{2 \zeta(t-1) \zeta(t)}{q^{t-1}} \left( \frac{2^t}{q} + 2^{t-1} + 1 \right).
    \end{split} 
    \end{equation*}
\end{itemize}
Chosing $t = 4$ and summing up the three estimates obtained in $(a)$, $(b)$ and $(c)$, we obtain:
\begin{equation*} 
    |\hat{u}_q|\left(1-\frac{2\pi^2}{3q^2}+\frac{\beta_0}{q^2}\right) \leq \frac{ \|u\|_4}{q^{4}}\left[ (\zeta(4)-1)\left(1-\frac{2\pi^2}{3q^2}+\frac{\beta_0}{q^2}\right) + \frac{2 \bar{C} \delta \zeta(3)\zeta(4)}{q} \left(\frac{16}{q} + 9 \right) + \delta^2 K_1 \zeta(2) + \delta K_2 \zeta(3)\right].
\end{equation*}
Since $\left(1-\frac{2\pi^2}{3q^2}+\frac{\beta_0}{q^2} \right) > 0.26$ for every $q \ge 3$, we finally get
\begin{equation*} 
|\hat{u}_q| \le \frac{ \|u\|_4}{q^{4}} \left( (\zeta(4) - 1) + \mathfrak{C} \delta \right) \qquad \forall q \ge 3
\end{equation*}
where
\begin{equation} \label{FRAK}
\mathfrak{C} := \frac{10 \bar{C}\zeta(3)\zeta(4) + K_1\zeta(2) + K_2\zeta(3)}{0.26}
\end{equation}
does not depend on $q \ge3$ and $\delta > 0$. Consequently, for $\delta > 0$ sufficiently small, it holds
\begin{equation*} 
|\hat{u}_q| < D \frac{ \|u\|_4}{q^{4}} \qquad \forall q \ge 3,
\end{equation*}
for a constant $D < 1$. Note that --by the normalization conditions-- we also have
$$ 
u(0)=\sum_{k=1}^\infty \hat{u}_k=0 \qquad \text{and} \qquad
u(0)+u\left(\frac{1}{2}\right)=2\sum_{k=1}^\infty \hat{u}_{2k} =0.
$$
From these it follows that also $|\hat{u}_2| < \frac{ \|u\|_4}{16} $ and $|\hat{u}_1| < \|u\|_4$. At the end, we conclude that $\|u\|_4=0$ and therefore $n_\gamma \equiv0$.
\end{proof}
\indent We are now ready to conclude the proof of Theorem \ref{THM3}. From the explicit estimates contained in the proof of the previous lemma, it follows that all the terms involved in the linear system (\ref{questa}) are $O\left(\frac{1}{q^4}\right)$. Differently, the term in the right member of (\ref{ordine}) is $O\left(\frac{1}{q^2}\right)$. Consequently, since the left members of (\ref{ordine}) and (\ref{questa}) are exactly the same, if $ u(s) =n_\gamma\left(0,s\right)\rho(t(s))^\frac{1}{3}$ solves (\ref{ordine}), then necessarily $\sum_{k \ge 1} \hat{u}_k (\beta_k + 2\pi i k \alpha_k) = 0$ and therefore (as a consequence of Lemma \ref{liposuzione}) $ u \equiv 0$. This means that $n_\gamma \equiv0$, which is the desired result. \\


\begin{thebibliography}{}
\bibitem{AlbTab} Albers, P.; Tabachnikov, S. Introducing symplectic billiards. Adv. Math.333 (2018), 822–867.
%\bibitem{AVILA} Avila, A.; De Simoi, J.; Kaloshin V. An integrable deformation of an ellipse of small eccentricity is an ellipse. Annals of Mathematics 184 (2016), 527–558.
%\bibitem{BMT}  Balestro, V.; Martini, H.; Teixeira, R. A new construction of Radon curves and related topics. Aequat. Math. 90, 1013–1024 (2016).
\bibitem{BaBe} Baracco, L.; Bernardi, O. Totally integrable symplectic billiards are ellipses. Adv. Math. 454 (2024), Paper No. 109873.
\bibitem{BaBeNa} Baracco, L.; Bernardi, O.; Nardi, A. Higher order terms of Mather's $\beta$-function for symplectic and outer billiards. J. Math. Anal. Appl.537 (2024), no.2, Paper No. 128353, 20 pp.
\bibitem{BaBeNa2} Baracco, L.; Bernardi, O.; Nardi, A. Bialy-Mironov type rigidity for centrally symmetric symplectic billiards. Nonlinearity 37 (2024), no.~12, Paper No. 125025, 12 pp. 
%\bibitem{BBT} Bialy, M; Bor, G.; Tabachnikov, S. Self-Bäcklund curves in centroaffine geometry and Lamé’s equation. Communications of the American Mathematical Society, 2 (2022), 232-282.
\bibitem{Bu}  Buchin S. Affine Differential Geometry. Science Press, Beijing, China, and Gordon and Breach, Science Publishers, Inc., New York, (1983).
\bibitem{DS} De Simoi, J.; Kaloshin, V.; Wei, Q. Dynamical spectral rigidity among $\mathbb{Z}_2$-symmetric strictly convex domains close to a circle. Appendix B coauthored with H. Hezari. Ann. of Math. (2)186 (2017), no.1, 277–314.
\bibitem{FSV} Fierobe, C.; Sorrentino, A.; Vig, A. Deformational spectral rigidity of axially symmetric symplectic billiards. Nonlinearity 38, No. 10 (2025).
\bibitem{Glu} Glutsyuk A. On infinitely many foliations by caustics in strictly convex open billiards. Ergodic Theory and Dynamical Systems. (2024);44(5):1418-1467.
%\bibitem{KaSo} Kaloshin, V; Sorrentino A. On the local Birkhoff conjecture for convex billiards. Annals of Mathematics, 188 (1): 315-380, 2018.
\bibitem{Laz} Lazutkin, V. F. Existence of caustics for the billiard problem in a convex domain. Mathematics of the USSR-Izvestiya, Volume 7, Number 1, 185-214 (1973).
\bibitem{Lud} Ludwig M. Asymptotic approximation of convex curves. Arch. Math 63, 377–384 (1994).
%\bibitem{MaSwa}  Martini, H.; Swanepoel, K.J. Antinorms and Radon curves. Aequationes Math. 71, 110–138 (2006).
%\bibitem{Martini}  Martini, H.; Swanepoel, K. Equiframed Curves– A Generalization of Radon Curves. Monatsh. Math. 141, 301–314 (2004).
\bibitem{MaMel} Marvizi S.; Melrose, R. Spectral invariants of convex planar regions. J. Diff. Geom. 17 (1982), 475–502.
\bibitem{PetSto} Petkov V.M.; Stoyanov, L.N. Geometry of Reflecting Rays and Inverse Spectral Problems, Pure Appl. Math., John Wiley \& Sons, Ltd., New York, (1992).
%\bibitem{Si} Siburg K.F. The principle of least action in geometry and dynamics. Lecture Notes in
%Mathematics, vol. 1844, xiii+ 128 pp. Berlin, Germany: Springer, (2004).
\bibitem{Sor} Sorrentino A. Computing Mather’s $\beta$-function for Birkhoff billiards. Discrete and Contin. Dyn. Syst. A, 35 (10): 5055-5082, (2015).
\bibitem{DT} Tsodikovich, D. Local Rigidity for Symplectic Billiards. J Geom Anal 35, 306 (2025).
% \bibitem{Tab} Tabachnikov S.; Doǧru F. Dual Billiards. The Mathematical Intelligencer 27: 18-25 (2005).
% \bibitem{Vig} Vig A. Compactness of Marked Length Isospectral Sets of Birkhoff Billiard Tables. Preprint
% https://arxiv.org/abs/2310.05426
% \bibitem{Zhang} Zhang J. Spectral invariants of convex billiard maps: a viewpoint of Mather's beta function.  arXiv:2009.12513v1 (2020).

\end{thebibliography}
\end{document}